\documentclass[a4paper,reqno]{amsart}

\usepackage{amsthm}
\usepackage{amsmath}
\usepackage{amssymb}
\usepackage{txfonts}
\usepackage[usenames,dvipsnames,svgnames,table]{xcolor}
\usepackage{textcomp}
\usepackage{microtype}
\usepackage{enumerate}
\usepackage{mathrsfs}
\usepackage{dsfont}
\usepackage{graphicx}
\usepackage{pgfplots}
	\pgfplotsset{compat=1.12}
\usepackage[font=small,labelfont=bf]{caption}
\usepackage[hidelinks]{hyperref}
\usepackage[foot]{amsaddr}

\theoremstyle{plain}
	\newtheorem{theorem}{Theorem}[section]
	\newtheorem{proposition}[theorem]{Proposition}
	\newtheorem{corollary}[theorem]{Corollary}

\newcommand{\e}{\operatorname{e}}
\newcommand{\rk}{\operatorname{rank}}
\newcommand{\ii}{\mathbf i \,}
\def\N{\mathbb N}
\def\R{\mathbb R}
\def\E{\mathcal E}
\def\I{\mathcal I}

\let\oldsqrt\sqrt
\def\sqrt{\mathpalette\DHLhksqrt}
\def\DHLhksqrt#1#2{%
\setbox0=\hbox{$#1\oldsqrt{#2\,}$}\dimen0=\ht0
\advance\dimen0-0.2\ht0
\setbox2=\hbox{\vrule height\ht0 depth -\dimen0}%
{\box0\lower0.4pt\box2}}

\allowdisplaybreaks

\begin{document}

\title{Measure-geometric Laplacians for discrete distributions}

\author{M.~Kesseb\"ohmer}
\address{Fachbereich 3 Mathematik, Universit\"at Bremen, Bibliothekstr. 1, 28359 Bremen, Germany.}
%

\author{T.~Samuel}
\address{Mathematics Department, California Polytechnic State University, San Luis Obispo, CA, USA.}
%

\author{H.~Weyer}
\address{Fachbereich 3 Mathematik, Universit\"at Bremen, Bibliothekstr. 1, 28359 Bremen, Germany.}
%


\begin{abstract}
In 2002 Freiberg and Z\"ahle introduced and developed a harmonic calculus for measure-geometric Laplacians associated to continuous distributions. We show their theory can be extended to encompass distributions with finite support and give a matrix representation for the resulting operators. In the case of a uniform discrete distribution we make use of this matrix representation to explicitly determine the eigenvalues and the eigenfunctions of the associated Laplacian.
\end{abstract}

\maketitle

\section{Introduction}\label{sec:introduction}

Motivated by the fundamental theorem of calculus, and based on the works of Feller \cite{Fe57} and Kac and Kre\u{\i}n \cite{KK68}, given an atomless Borel probability measure $\mu$ supported on a compact subset of $\R$, Freiberg and Z\"{a}hle \cite{FZ02} introduced a \mbox{measure-geometric} approach to define a first order differential operator $\nabla^{\mu}$ and a second order differential operator \mbox{$\Delta^{\mu, \mu} \coloneqq \nabla^{\mu} \circ \nabla^{\mu}$}, with respect to $\mu$.  In the case that $\mu$ is the Lebesgue measure, it was shown that $\nabla^{\mu}$ coincides with the weak derivative.  Moreover, a harmonic calculus for $\Delta^{\mu, \mu}$ was developed and, when $\mu$ is a self-similar measure supported on a Cantor set, the authors proved the eigenvalue counting function of $\Delta^{\mu, \mu}$ is comparable to the square-root function. In \cite{KSW16} for continuous measures the exact eigenvalues and eigenfunctions were obtained and it was shown the eigenvalues do not depend on the given measure. Arzt \cite{A15b} has also considered the \mbox{Kre\u{\i}n-Feller} operator $\Delta^{\mu, \Lambda} \coloneqq \nabla^{\mu} \circ \nabla^{\Lambda}$, where $\mu$ denotes a continuous Borel probability measure and $\Lambda$ denotes the Lebesgue measure, see \cite{Fr05, Fu87} for further results in this direction.

Here, we show this framework can be extended to included purely atomic measures $\mu$. Unlike in the case when one has a measure with a continuous distribution function (see for instance \cite{FZ02,KSW16}), we prove the operators $\nabla^{\mu}$ and $\Delta^{\mu,\mu}$ are no longer symmetric. To circumvent this problem, we consider the operator $\nabla^{\mu}$, its adjoint $(\nabla^{\mu})^{*}$ and define the \mbox{$\mu$-Laplacian} to be $\Delta^{\mu} = -(\nabla^{\mu})^{*} \circ \nabla^{\mu}$. We give matrix representations for these operators, noting they coincide with the normalised Laplacian matrix of a cycle graph \cite{Bi93} and resemble a discretisation of a one-dimensional Laplacian on a non-uniform grid \cite{Le07}. Further, we discuss properties of the eigenvalues and eigenfunctions of the operator $\Delta^{\mu}$.  In particular, we show the eigenfunctions for distributions with finite support are not necessarily of the form $f_{\kappa}^{\mu}(\cdot) \coloneqq \sin(\pi \kappa F_{\mu}(\cdot))$ or $g_{\kappa}^{\mu}(\cdot) \coloneqq \cos(\pi \kappa F_{\mu}(\cdot))$, for $\kappa \in \R \setminus \{ 0 \}$ and where $F_{\mu}$ denotes the distribution function of $\mu$.  This differs from the case of continuous distributions, see \cite{FZ02,KSW16,Z05}. Additionally, in the case that $\mu$ is a uniform discrete probability distribution we explicitly determine the eigenvalues and eigenfunctions of $\Delta^{\mu}$.

\newpage

\subsection*{Outline} 

In Section \ref{sec:definition_and_properties} we present necessary definitions and basic properties of $\nabla^{\mu}$, $(\nabla^{\mu})^{*}$ and $\Delta^{\mu}$ and give matrix representations for these operators.  In Section \ref{sec:eigenvalues} we prove general results concerning the spectral properties of $\Delta^{\mu}$.  We conclude with Section \ref{sec:examples}, where explicit computations are carried out when $\mu$ is a uniform discrete probability distribution.

\section{Definitions and analytic properties of \texorpdfstring{$\Delta^{\mu}$}{TEXT}}\label{sec:definition_and_properties}

Set $\I \coloneqq [0,1]$ and let $\delta_z$ denote the Dirac-measure at $z$, for some fixed $z \in \I$. Let $\mu$ denote the probability measure $\mu \coloneqq \sum_{i=1}^{N}\alpha_i \delta_{z_i}$, where $N \in \N$, $0 \leq z_{1} < z_{2} < \dots < z_{N} < 1$ and $\alpha_{i} > 0$, for $i\in\{1,\dots,N\}$. We denote the set of real-valued square-integrable functions on $\I$ by $\mathfrak{L}^{2}_{\mu} = \mathfrak{L}^{2}_{\mu}(\I)$, we define $\mathcal{N}_{\mu}(\I)$ to be set of $\mathfrak{L}^{2}_{\mu}$-functions which are constant zero $\mu$-almost everywhere, and we let $L^{2}_{\mu} = L^{2}_{\mu}(\I) \coloneqq \mathfrak{L}^{2}_{\mu}(\I) \setminus \mathcal{N}_{\mu}(\I)$. The latter space is a finite-dimensional inner product space with inner product $\langle \cdot, \cdot \rangle$ given by 
\begin{align*}
\langle f, g \rangle = \langle f, g \rangle_{\mu} \coloneqq \sum_{i=1}^{N} \alpha_i  f(z_i)  g(z_i).
\end{align*}
We define the set of \textit{$\mu$-differentiable functions on $\I$ with periodic boundary conditions} by
\begin{align}\label{Def:D1line}
\begin{aligned}
\mathscr{D}_{\mu}^{1} = \mathscr{D}_{\mu}^{1}(\I) \coloneqq \bigg\{ f \in \mathfrak{L}^2({\mu}) \colon \text{there exists} \, f' \in L^2_{\mu} \; \text{such that} \; f(0) = f(1) \; \text{and}& \\
f(x) = f(0) + \int \mathds{1}_{[0,x)} f' \, \mathrm{d}\mu \; \text{for all} \; x \in \I &
\bigg\},
\end{aligned}
\end{align}
where we understand $[0,0) = \emptyset$. Note, the function $f'$ defined in \eqref{Def:D1line} is unique in $L^2_{\mu}$.  Since $f(0) = f(1) = f(0) + \int \mathds{1}_{[0,1)} f' \, \mathrm{d}\mu$, it follows that
\begin{align}\label{cor:int_derivative_0}
 \int \mathds{1}_{[0,1)} f' \, \mathrm{d}\mu = 0.
\end{align}
For $f \in \mathscr{D}_{\mu}^{1}$ and $f'$ as in \eqref{Def:D1line}, the operator $\nabla^{\mu} \colon \mathscr{D}_{\mu}^{1} \to L^2_{\mu}$ defined by $\nabla^{\mu}f \coloneqq f'$ is called the \mbox{\textit{${\mu}$-derivative}}.  Linearity of the integral yields $\nabla^{\mu}$ is linear on $\mathscr{D}_{\mu}^{1}$.  As $\mu$ is a linear combination of Dirac measures, we can reformulate the defining equation of $\nabla^{\mu}f$ given in \eqref{Def:D1line} by
\begin{align}\label{eq:def_nabla_f}
f(x) = f(0) + \sum_{\substack{i \in \{1, \dots, N \}\\z_{i} < x}} \alpha_{i}\nabla^{\mu}f(z_i),
\end{align}
where $f \in \mathscr{D}_{\mu}^{1}$ and $x \in \I$. Thus, if $\mu$ is a Dirac measure, that is $N = 1$, then \eqref{cor:int_derivative_0} becomes $\alpha_1\nabla^{\mu}f(z_1) = 0$. Hence, from \eqref{eq:def_nabla_f}, it follows that $\mathscr{D}_{\mu}^{1}$ is the set of constant functions. In other words, the operator $\nabla^{\mu}$ is the null-operator, and so, from here on we assume $N \geq 2$.

The periodic boundary conditions and \eqref{eq:def_nabla_f} together imply that a function $f \in \mathscr{D}_{\mu}^{1}$ is piecewise constant; namely, $f\vert_{[0,z_{1}] \cup (z_{N},1]}$ and $f\vert_{(z_i,z_{i+1}]}$ are constant, for $i \in \{ 1, \dots, N-1 \}$.  Therefore, $f$ is uniquely determined by the vector $(f(z_1),\dots,f(z_N))^{\top}$, and thus, there exists an $N \times N$-matrix $A$ with
\begin{align*}
A ( f(z_{1}), \dots, f(z_{N}))^{\top} = (\nabla^{\mu} f(z_{1}), \dots, \nabla^{\mu} f(z_{N}))^{\top}\hspace{-0.5em}.
\end{align*}
From \eqref{eq:def_nabla_f} and the fact that $f(1)=f(0)=f(z_1)$, we have  
\begin{align*}
	\nabla^{\mu}f(z_N) &= \cfrac{f(z_{1})-f(z_{N})}{\alpha_{N}} \quad \text{and} \quad
	\nabla^{\mu}f(z_{n})=\cfrac{f(z_{n+1})-f(z_{n})}{\alpha_{n}},
\end{align*}
for $n\in\{1,\dots,N-1\}$, and hence,
\begin{align*}
	A =
		\begin{pmatrix}
			-\alpha_{1}^{-1} &  \alpha_{1}^{-1} & 0 & \cdots & 0 & 0 & 0\\[0.25em]
			0 & -\alpha_{2}^{-1} &  \alpha_{2}^{-1} & \cdots & 0 & 0 & 0\\[0.25em]
			0 & 0 & - \alpha_{3}^{-1}  & \cdots & 0 & 0 & 0\\[0.25em]
			\vdots & \vdots & \vdots &  \ddots & \vdots & \vdots & \vdots\\
			0 & 0 & 0 &  \cdots & -\alpha_{N - 2}^{-1} & \alpha_{N-2}^{-1} & 0\\[0.25em]
			0 & 0 & 0 & \cdots & 0 & -\alpha_{N-1}^{-1} & \alpha_{N-1}^{-1}\\[0.25em]
			\alpha_{N}^{-1} & 0 & 0 & \cdots & 0 & 0 & -\alpha_{N}^{-1}
		\end{pmatrix}.
\end{align*}
Since $\sum_{i=1}^{N} \alpha_{i} g^2(z_i) < \infty$, for all $g \in \mathscr{D}_{\mu}^{1}$, there exists a natural embedding $\pi \colon \mathscr{D}_{\mu}^{1} \to L_{\mu}^{2}$. In fact, from the matrix representation given above it follows that $\pi(\mathscr{D}_{\mu}^{1}) = L_{\mu}^{2}$. In other words, every equivalence class of  $L_{\mu}^{2}$ has a $\mu$-differentiable representative, and so, from here on we will not distinguish between $\mathscr{D}_{\mu}^{1}$ and $\pi(\mathscr{D}_{\mu}^{1})$.

Notice, $A$ is not self-adjoint, and $A^2$ is self-adjoint if and only if $N = 2$ and $\alpha_{1} = \alpha_{2}$. Hence, the operators $\nabla^{\mu}$ and $\Delta^{\mu} \coloneqq \nabla^{\mu} \circ \nabla^{\mu}$ are not in general self-adjoint.  To obtain a self-adjoint operator we follow the program of Kigami \cite{Ki93, Ki01, Ki03}, and Kigami and Lapidus \cite{KL01}, and use the bilinear form $\E$ defined by 
\begin{align*}
\E(f,g) = \E^{\mu}(f,g) \coloneqq \langle \nabla^{\mu}f,\nabla^{\mu}g \rangle,
\end{align*}
for $f, g \in \mathscr{D}_{\mu}^{1}$.  We refer to $\E$ as the \textit{$\mu$-energy form}.
 
\begin{theorem}
The $\mu$-energy form $\E$ is a Dirichlet form.
\end{theorem}

\begin{proof} 
The $\mu$-energy form is bilinear since the inner product is bilinear, $\nabla^{\mu}$ is linear and every equivalence class of  $L_{\mu}^{2}$ has a $\mu$-differentiable representative. The symmetry and the non-negativity of $\E$ follow from the properties of the inner product. For every $f \in \mathscr{D}_{\mu}^{1}$, the function $\hat{f} \colon \I \to \R$, defined by $\hat{f}(x) \coloneqq \min(\max(f(x),0),1)$, belongs to $\mathscr{D}_{\mu}^{1}$, and as 
\begin{align*}
\lvert \hat{f}(z_{i+1}) - \hat{f}(z_{i}) \rvert \leq \lvert f(z_{i+1}) - f(z_{i}) \rvert
\quad \text{and} \quad
\lvert \hat{f}(z_{1}) - \hat{f}(z_{N}) \rvert \leq \lvert f(z_{1}) - f(z_{N}) \rvert,
\end{align*}
it follows that $\E(\hat{f}, \hat{f}) \leq \E(f,f)$. The properties of $\langle \cdot, \cdot \rangle $ and $\E$ yield $\mathscr{D}_{\mu}^{1}$ equipped with $\langle \cdot, \cdot \rangle_{\E} \coloneqq \langle \cdot, \cdot \rangle + \E(\cdot, \cdot)$ is an inner product space. Now, for every Cauchy sequence $(f_{n})_{n \in \N}$ in $(\mathscr{D}_{\mu}^{1}, \langle \cdot, \cdot \rangle_{\E})$, we have both  $(f_{n})_{n \in \N}$ and $(\nabla^{\mu}f_{n})_{n \in \N}$ are Cauchy-sequences in $L_{\mu}^{2}$. Hence, there exist $\tilde{f_{0}}, \tilde{f_{1}} \in L_{\mu}^{2}$ with $\lim_{n\to \infty} \lVert f_n - \tilde{f_{0}} \rVert = 0$ and $\lim_{n\to \infty} \lVert \nabla^{\mu}f_n - \tilde{f_{1}} \rVert = 0$, where $\lVert f \rVert^2 \coloneqq \langle f,f \rangle$, for $f \in L_{\mu}^{2}$. Since $\nabla^{\mu}$ is linear, it is continuous, and so $\nabla^{\mu} \tilde{f_{0}} = \tilde{f_{1}}$.  This implies that $\lim_{n \to \infty}f_{n} = \tilde{f_{0}} \in \mathscr{D}_{\mu}^{1}$, with respect to $\langle \cdot, \cdot \rangle_{\E}$.
\end{proof}

Notice, since the analysis reduces to a finite dimensional vector space, the $\mu$-energy form is also a graph energy form as treated in \cite{KL12}.

We say that $f \in \mathscr{D}_{\mu}^{1}$ belongs to $\mathscr{D}_{\mu}^{2} = \mathscr{D}_{\mu}^{2}(\I)$, if there exists a $h \in L^2_{\mu}$, necessarily unique, such that $\E(f,g) = -\langle h,g \rangle$, for all $g \in \mathscr{D}_{\mu}^{1}$. We define the \textit{$\mu$-Laplacian} to be the operator $\Delta^{\mu} \colon \mathscr{D}_{\mu}^{2} \to L^2_{\mu}$ given by $\Delta^{\mu} f \coloneqq h$.  Indeed, for an arbitrary $g \in \mathscr{D}_{\mu}^{1}$, we observe
\begin{align}\label{eq:energy_identity}
\langle \nabla^{\mu}f,\nabla^{\mu}g \rangle  = -\langle \Delta^{\mu} f,g \rangle, \quad \text{and thus,} \quad \Delta^{\mu} = -\left(\nabla^{\mu}\right)^{*} \circ \nabla^{\mu}.
\end{align}
With this, we conclude $B \coloneqq -A^{\top}A$ is a matrix representation of $\Delta^{\mu}$; in fact, for $N=2$, 
\begin{align}\label{eq:matrix_N=2}
B = \begin{pmatrix}
	-\alpha_{1}^{-2} - \alpha_{2}^{-2} & \alpha_{1}^{-2} + \alpha_{2}^{-2}\\[0.25em]
	\alpha_{1}^{-2} + \alpha_{2}^{-2}  & -\alpha_{1}^{-2} - \alpha_{2}^{-2}
\end{pmatrix},
\end{align}
and, for $N \geq 3$, we have $B$ is the $N \times N$-matrix

\noindent\resizebox{\linewidth}{!}{%
$\begin{pmatrix}
-\alpha_{N}^{-2} - \alpha_{1}^{-2} & \alpha_{1}^{-2} & 0 & \cdots & 0 & 0 & \alpha_{N}^{-2}\\[0.25em]
\alpha_{1}^{-2} & -\alpha_{1}^{-2} - \alpha_{2}^{-2} &  \alpha_{2}^{-2} & \cdots & 0& 0 & 0 \\[0.25em]
0 & \alpha_{2}^{-2} & -\alpha_{2}^{-2} - \alpha_{3}^{-2} & \cdots & 0 & 0 & 0 \\[0.25em]
\vdots & \vdots & \vdots & \ddots & \vdots & \vdots & \vdots\\[0.25em]
0 & 0 & 0 & \cdots & -\alpha_{N-3}^{-2} - \alpha_{N-2}^{-2} & \alpha_{N-2}^{-2} & 0\\[0.25em]
0 & 0 & 0 & \cdots & \alpha_{N-2}^{-2} & -\alpha_{N-2}^{-2} - \alpha_{N-1}^{-2} &  \alpha_{N-1}^{-2}\\[0.25em]
\alpha_{N}^{-2} & 0 & 0 & \cdots & 0 & \alpha_{N-1}^{-2} & -\alpha_{N-1}^{-2} - \alpha_{N}^{-2}
\end{pmatrix}$.
}

\begin{theorem}\label{Thm:mixed_sym_non-pos}
The operator $\Delta^{\mu}$ is linear, self-adjoint and non-positive.
\end{theorem}

\begin{proof}
Linearity follows from linearity of $\nabla^{\mu}$ and bilinearity of $\E$. Self-adjointness is a consequence of symmetry of $\E$.  An application of \eqref{eq:energy_identity} yields $\langle \Delta^{\mu} f,f \rangle = -\langle \nabla^{\mu}f,\nabla^{\mu}f \rangle \leq 0$, and hence, $\Delta^{\mu}$ is non-positive.
\end{proof}

The above demonstrates one can view the operator $\Delta^{\mu}$ as a Laplacian matrix of a weighted cycle graph, see \cite{Bi93,Ch97}. Indeed, the analytic results from spectral graph theory may be carried over to our setting, which would certainly be interesting to investigate and could lead to further results concerning $\Delta^{\mu}$.  Moreover, the matrix $B$ resembles the matrices appearing in the finite difference methods used in numerical analysis of ordinary differential equations\footnote{\;\parbox[t]{0.95\linewidth}{We would like to thank Paul Choboter, Maik Gr\"oger, Jens Rademacher and Alfred Schmidt for bringing this to our attention.}}, see \cite{Le07}.  We also observe the matrix representation of the operator $\Delta^{\mu}$ only depends on the order and the weighting of the atoms of $\mu$ and is independent on the distances between $z_{i}$ and $z_{j}$, for all $i, j \in \{ 1, \dots, N \}$.

\section{Spectral properties of \texorpdfstring{$\Delta^{\mu}$}{TEXT}}\label{sec:eigenvalues}

By definition of the matrix $B$, to find the eigenvalues and eigenfunctions of $\Delta^{\mu}$, it suffices to compute the eigenvalues and eigenvectors of $B$. 

\begin{proposition}\label{prop:lower_bound}
If $\lambda$ is an eigenvalue of $\Delta^{\mu}$, then $\lambda \in \R$ and $\displaystyle 2\hspace{-0.5em}\min_{i \in \{ 1, \dots, N\}} \hspace{-0.25em} B_{i,i} \leq \lambda \leq 0$.
\end{proposition}

\begin{proof}
Theorem \ref{Thm:mixed_sym_non-pos} together with the fact that $B$ is self-adjoint with entries in $\R$ yields all eigenvalues are non-positive real numbers. Since the spectral norm is bounded above by the column-sum norm, given an eigenvalue $\lambda$ of $B$, it follows that 
\[
\lvert \lambda \rvert \leq 2 \max \{ \lvert B_{i,i} \rvert \colon i \in \{ 1, \dots, N \} \}. \qedhere
\]
\end{proof}

\begin{proposition}\label{prop:eigenvalue zero}
The operator $\Delta^{\mu}$ has a simple eigenvalue at $\lambda = 0$ where the corresponding eigenfunction is the constant function with value $1$.
\end{proposition}

\begin{proof}
A direct calculation reveals $B(1,\dots,1)^{\top}=(0,\dots,0)^{\top}$, and hence, we have that \mbox{$\lambda = 0$} is an eigenvalue of $\Delta^{\mu}$ where the constant function with value $1$ is the corresponding eigenfunction. A row reduced echelon form of $A$ is an upper triangular matrix with a single zero on the diagonal, and so $\rk(A)=N-1$. Combining this with the fact that $\rk(B)=\rk(A^{\top}A)=\rk(A)$ it follows that the eigenvalue $\lambda = 0$ of $B$ is simple.
\end{proof}

We observe, for $N=2$, that the eigenvalues of $\Delta^{\mu}$ are $\lambda_{0} = 0$ and $\lambda_{1} = -2 (\alpha_{1}^{-2} + \alpha_{2}^{-2})$ with corresponding eigenfunctions,
\begin{align*}
f_{0}(x) = 1, \; \text{for all $x \in [0,1]$,}\quad \text{and} \quad f_{1} (x) = \begin{cases} \phantom{-}1 \; &\text{for} \; x\in [0,z_{1}] \cup (z_{2},1], \\
-1 \; & \text{otherwise.}
\end{cases}
\end{align*}
This follows since $\lambda_0$ and $\lambda_{1}$ are eigenvalues of the matrix $B$, given in \eqref{eq:matrix_N=2}, with corresponding eigenvectors $v^{(0)}=  (1,1)^{\top}$ and $v^{(1)}= (1,-1)^{\top}$, respectively. In particular, we see that, in this case, the lower bound in Theorem \ref{prop:lower_bound} is sharp. 
 
Different to the case of continuous distributions, the eigenfunctions for distributions with finite support are not necessarily of the form 
\begin{align*}
f_{\kappa}^{\mu}(x)=\sin(\pi \kappa F_{\mu}(x))
\quad \text{or} \quad 
g_{\kappa}^{\mu}(x)=\cos(\pi \kappa F_{\mu}(x)),
\end{align*}
for $x\in[0,1]$ and $\kappa \in \R \setminus \{ 0 \}$, where $F_{\mu}$ denotes the distribution function of $\mu$, which we now address in the following paragraph.

Let $m_{1}$ and $m_{2}$ denote two positive real numbers with $3 m_{1} + 3 m_{2} = 1$, and set $r = m_{2}/m_{1}$.  Consider the discrete distribution $\mu = \sum_{i=1}^{6} \alpha_{i} \delta_{z_i}$, where $0 < z_{1} < z_{2} < \cdots < z_{6} < 1$, $\alpha_{1} = \alpha_{3} = \alpha_{5} = m_{1}$ and $\alpha_{2} = \alpha_{4} = \alpha_{6} = m_{2}$.  A direct calculation shows the eigenvalues for the matrix representation of $\Delta^{\mu}$ are 
\begin{align*}
\lambda_{0} &= 0, & \lambda_{1} = \lambda_{5} &= - (m_1^{-2} + m_2^{-2}) + \sqrt{m_1^{-4} + m_2^{-4} - m_1^{-2}m_2^{-2}},\\
\lambda_{3} &=-2(m_{1}^{-2} + m_{2}^{-2}), & \lambda_{2} = \lambda_{4} &= - (m_1^{-2} + m_2^{-2}) - \sqrt{m_1^{-4} +m_2^{-4} - m_1^{-2}m_2^{-2}},
\end{align*}
with corresponding eigenvectors
\begin{align*}
v^{(0)} &= \left( 1, 1, 1, 1, 1, 1 \right)^{\top}\hspace{-0.5em},\\
v^{(1)} &= \left( r^2, \sqrt{1 + r^4 - r^2}, 1- r^2, - \sqrt{1 + r^4 - r^2},  -1, 0 \right)^{\top}\hspace{-0.5em},\\
v^{(2)} &= \left( r^2, -\sqrt{1 + r^4 - r^2}, 1-r^2, \sqrt{1 + r^4 - r^2}, -1, 0 \right)^{\top}\hspace{-0.5em},\\
v^{(3)} &= \left(1, -1, 1, -1, 1, -1\right)^{\top}\hspace{-0.5em},\\
v^{(4)} &= \left(\sqrt{1 + r^4 - r^2}, 1 - r^2, \: -\sqrt{1 + r^4 - r^2}, r^2, 0, -1\right)^{\top}\hspace{-0.5em},\\
v^{(5)} &= \left(\sqrt{1 + r^4 - r^2}, r^2-1, \: -\sqrt{1 + r^4 - r^2}, -r^2, 0, 1\right)^{\top}\hspace{-0.5em}.
\end{align*}
For the eigenvalues with multiplicity two, namely $\lambda_{1} = \lambda_{5}$ and $\lambda_{2} =\lambda_{4}$, notice that the sets of tuples
\begin{align*}
S_{1, 5} \coloneqq \left\{ \left( v^{(1)}_{1}, v^{(5)}_{1} \right), \dots, \left(v^{(1)}_{6}, v^{(5)}_{6}\right) \right\}
\quad \text{and} \quad
S_{2,4} \coloneqq \left\{ \left(v^{(2)}_{1}, v^{(4)}_{1}\right), \dots, \left(v^{(2)}_{6}, v^{(4)}_{6}\right) \right\}
\end{align*}
determine the same ellipse, namely,
\begin{align*}
\sqrt{1+r^{-4} - r^{-2}} \left(x^2+y^2-1\right) = \left( 2-r^{-2} \right)xy,
\end{align*}
which is non-axisymmetric; see Figure \ref{fig:Ellipses} for an example and compare with Theorem \ref{prop:ev_all_equal} and Corollary \ref{prop:ev_and_ef_of_Delta_all_equal}, where the analogous set of tuples lie on the unit circle. This latter property demonstrates that the eigenspace for $\lambda_{1} = \lambda_{5}$  (or $\lambda_{2} =\lambda_{4}$) is not spanned by $\{ f_{\kappa_{1}}^{\mu}, g_{\kappa_{2}}^{\mu} \}$ for any $\kappa_{1}, \kappa_{2} \in \R$.  Moreover, in this explicit case, setting 
\begin{align*}
w^{(\kappa)} \coloneqq \left( \sin(\pi \kappa F_{\mu}(z_{1}) ), \dots, \sin( \pi \kappa F_{\mu}(z_{6}) ) \right)^{\top}\hspace{-0.5em},
\quad
u^{(\kappa)} \coloneqq \left( \cos(\pi \kappa F_{\mu}(z_{1})), \dots, \cos(\pi \kappa F_{\mu}(z_{6}) ) \right)^{\top}\hspace{-0.5em},
\end{align*}
a direct calculation shows $B w^{(\kappa)} \neq \lambda_{i} w^{(\kappa)}$ and $B u^{(\kappa)} \neq \lambda_{i} u^{(\kappa)}$, for all $\kappa \in \R \setminus \{ 0 \}$ and all \mbox{$i \in \{ 0, \dots, 5 \}$}.  This latter result also holds when replacing $F_{\mu}(z_{i})$ by $\widetilde{F}_{\mu}(z_{i}) \coloneqq F_{\mu}(z_{i} - \varepsilon)$, for a fixed $\varepsilon \in ( 0, \min\{ z_{1}, \min \{ z_{i + 1} - z_{i} \colon i \in \{ 1, \dots, 5\} \}  \} )$.  For further details and examples, we refer the reader to \cite{We18}.

For the case $m_{1} = 1/4$ and $m_{2} = 1/12$, in Figure \ref{fig:Counterexapmle} the corresponding eigenfunctions of the operator $\Delta^{\mu}$ are sketched and Figure \ref{fig:Ellipses} illustrates the point plots of $S_{1, 5}$ and $S_{2, 4}$ together with the corresponding ellipse.

\begin{figure}[ht!]
\centering
\resizebox{0.85\linewidth}{!}{%
\begin{tikzpicture}
  \draw[->] (-0.1,0) -- (4,0);
  \draw[->] (0,-2.75) -- (0,2.75);
  \draw[thick,smooth] (0.25,-0.05) -- (0.25,0.05) node[below] {$z_{1}$};  
  \draw[thick,smooth] (1,-0.05) -- (1,0.05) node[below] {$z_{2}$};  
  \draw[thick,smooth] (1.4,-0.05) -- (1.4,0.05) node[below] {$z_{3}$};  
  \draw[thick,smooth] (2,-0.05) -- (2,0.05) node[below] {$z_{4}$};  
  \draw[thick,smooth] (2.75,-0.05) -- (2.75,0.05) node[below] {$z_{5}$};  
  \draw[thick,smooth] (3.25,-0.05) -- (3.25,0.05) node[below] {$z_{6}$};  
  \draw[thick,smooth] (3.5,-0.05) -- (3.5,0.05) node[below] {$1$};   
    \draw[thick,smooth] (-0.05,2) -- (0.05,2) node[left] {$1$};  
    \draw[thick,smooth] (-0.05,-2) -- (0.05,-2) node[left] {$-1$}; 
  \draw[thick,domain=0:0.75,smooth,variable=\w] plot ({\w},2);
  \draw[thick,domain=0.75:1.25,smooth,variable=\x]  plot ({\x},2);
  \draw[thick,domain=1.25:2.5,smooth,variable=\y]  plot ({\y},2);
  \draw[thick,domain=2.5:3.5,smooth,variable=\z]  plot ({\z},2);
%
  \draw[dashed,smooth] (0.25,0) -- (0.25,2);
  \draw[dashed,smooth] (1,0) -- (1,2);
  \draw[dashed,smooth] (1.4,0) -- (1.4,2);
  \draw[dashed,smooth] (2,0) -- (2,2);
  \draw[dashed,smooth] (2.75,0) -- (2.75,2);
  \draw[dashed,smooth] (3.25,0) -- (3.25,2);
  \filldraw (0.25,2) circle (1.4pt);
  \filldraw (1,2) circle (1.4pt);
  \filldraw (1.4,2) circle (1.4pt);
  \filldraw (2,2) circle (1.4pt);
  \filldraw (2.75,2) circle (1.4pt);
  \filldraw (3.25,2) circle (1.4pt);
\end{tikzpicture}
\qquad
\begin{tikzpicture}
  \draw[->] (-0.1,0) -- (4,0);
  \draw[->] (0,-2.75) -- (0,2.75);
  \draw[thick,smooth] (0.25,-0.05) -- (0.25,0.05) node[below] {$z_{1}$};  
  \draw[thick,smooth] (1,-0.05) -- (1,0.05) node[below] {$z_{2}$};  
  \draw[thick,smooth] (1.4,-0.05) -- (1.4,0.05) node[below] {$z_{3}$};  
  \draw[thick,smooth] (2,-0.05) -- (2,0.05) node[below] {$z_{4}$};  
  \draw[thick,smooth] (2.75,-0.05) -- (2.75,0.05) node[below] {$z_{5}$};  
  \draw[thick,smooth] (3.25,-0.05) -- (3.25,0.05) node[below] {$z_{6}$};  
  \draw[thick,smooth] (3.5,-0.05) -- (3.5,0.05) node[below] {$1$};   
    \draw[thick,smooth] (-0.05,2) -- (0.05,2) node[left] {$1$};  
    \draw[thick,smooth] (-0.05,-2) -- (0.05,-2) node[left] {$-1$}; 
  \draw[thick,domain=0:0.25,smooth,variable=\a] plot ({\a},0.222);
  \draw[thick,domain=0.31:1,smooth,variable=\b]  plot ({\b},1.898);
  \draw[thick,domain=1.06:1.4,smooth,variable=\c]  plot ({\c},1.777);
  \draw[thick,domain=1.46:2,smooth,variable=\d]  plot ({\d},-1.898);
  \draw[thick,domain=2.06:2.75,smooth,variable=\e]  plot ({\e},-2);
  \draw[thick,domain=2.81:3.25,smooth,variable=\f]  plot ({\f},0);
  \draw[thick,domain=3.31:3.5,smooth,variable=\g]  plot ({\g},0.222);
  \filldraw (0.25,0.222) circle (1.4pt);
  \draw (0.25,1.898) circle (1.4pt);
  \filldraw (1,1.898) circle (1.4pt);
  \draw (1,1.777) circle (1.4pt);
  \filldraw (1.4,1.777) circle (1.4pt);
  \draw (1.4,-1.898) circle (1.4pt);
  \filldraw (2,-1.898) circle (1.4pt);
  \draw (2,-2) circle (1.4pt);
  \filldraw (2.75,-2) circle (1.4pt);
  \draw (2.75,0) circle (1.4pt);
  \filldraw (3.25,0) circle (1.4pt);
  \draw (3.25,0.222) circle (1.4pt);
  \draw[dashed,smooth] (0.25,0) -- (0.25,1.838);
  \draw[dashed,smooth] (0.25,0.162) -- (0.25,0);  
  \draw[dashed,smooth] (1,0) -- (1,1.717);
  \draw[dashed,smooth] (1.4,1.717) -- (1.4,0);
  \draw[dashed,smooth] (1.4,-0.35) -- (1.4,-1.717);
  \draw[dashed,smooth] (2,-0.35) -- (2,-1.717);
  \draw[dashed,smooth] (2.75,-0.35) -- (2.75,-1.94);
  \draw[dashed,smooth] (3.25,0) -- (3.25,0.162);
\end{tikzpicture}
\qquad
\begin{tikzpicture}
  \draw[->] (-0.1,0) -- (4,0); 
  \draw[->] (0,-2.75) -- (0,2.75);
  \draw[thick,smooth] (0.25,-0.05) -- (0.25,0.05) node[below] {$z_{1}$};  
  \draw[thick,smooth] (1,-0.05) -- (1,0.05) node[below] {$z_{2}$};  
  \draw[thick,smooth] (1.4,-0.05) -- (1.4,0.05) node[below] {$z_{3}$};  
  \draw[thick,smooth] (2,-0.05) -- (2,0.05) node[below] {$z_{4}$};  
  \draw[thick,smooth] (2.75,-0.05) -- (2.75,0.05) node[below] {$z_{5}$};  
  \draw[thick,smooth] (3.25,-0.05) -- (3.25,0.05) node[below] {$z_{6}$};  
  \draw[thick,smooth] (3.5,-0.05) -- (3.5,0.05) node[below] {$1$};   
    \draw[thick,smooth] (-0.05,2) -- (0.05,2) node[left] {$1$};  
    \draw[thick,smooth] (-0.05,-2) -- (0.05,-2) node[left] {$-1$}; 
  \draw[thick,domain=0:0.25,smooth,variable=\a] plot ({\a},0.222);
  \draw[thick,domain=0.31:1,smooth,variable=\b]  plot ({\b},-1.898);
  \draw[thick,domain=1.06:1.4,smooth,variable=\c]  plot ({\c},1.777);
  \draw[thick,domain=1.46:2,smooth,variable=\d]  plot ({\d},1.898);
  \draw[thick,domain=2.06:2.75,smooth,variable=\e]  plot ({\e},-2);
  \draw[thick,domain=2.81:3.25,smooth,variable=\f]  plot ({\f},0);
  \draw[thick,domain=3.31:3.5,smooth,variable=\g]  plot ({\g},0.222);
  \filldraw (0.25,0.222) circle (1.4pt);
  \draw (0.25,-1.898) circle (1.4pt);
  \filldraw (1,-1.898) circle (1.4pt);
  \draw (1,1.777) circle (1.4pt);
  \filldraw (1.4,1.777) circle (1.4pt);
  \draw (1.4,1.898) circle (1.4pt);
  \filldraw (2,1.898) circle (1.4pt);
  \draw (2,-2) circle (1.4pt);
  \filldraw (2.75,-2) circle (1.4pt);
  \draw (2.75,0) circle (1.4pt);
  \filldraw (3.25,0) circle (1.4pt);
  \draw (3.25,0.222) circle (1.4pt);
  \draw[dashed,smooth] (0.25,0.162) -- (0.25,0);
  \draw[dashed,smooth] (0.25,-0.35) -- (0.25,-1.838);  
  \draw[dashed,smooth] (1,-0.35) -- (1,-1.838);
  \draw[dashed,smooth] (1,0) -- (1,1.717);
  \draw[dashed,smooth] (1.4,1.717) -- (1.4,0);
  \draw[dashed,smooth] (2,0) -- (2,1.838);
  \draw[dashed,smooth] (2,-0.35) -- (2,-1.94);
  \draw[dashed,smooth] (2.75,-0.35) -- (2.75,-1.94);
  \draw[dashed,smooth] (3.25,0) -- (3.25,0.162);
   \end{tikzpicture}}
\end{figure}

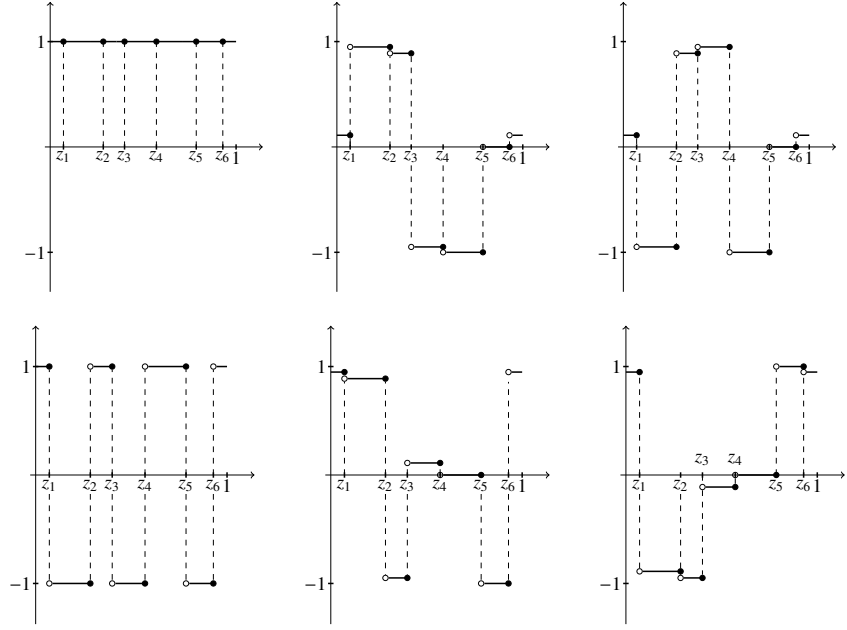
\begin{figure}[ht!]
\centering
\resizebox{0.875\linewidth}{!}{%
\begin{tikzpicture}
  \draw[->] (-0.1,0) -- (4,0);
  \draw[->] (0,-2.75) -- (0,2.75);
  \draw[thick,smooth] (0.25,-0.05) -- (0.25,0.05) node[below] {$z_{1}$};  
  \draw[thick,smooth] (1,-0.05) -- (1,0.05) node[below] {$z_{2}$};  
  \draw[thick,smooth] (1.4,-0.05) -- (1.4,0.05) node[below] {$z_{3}$};  
  \draw[thick,smooth] (2,-0.05) -- (2,0.05) node[below] {$z_{4}$};  
  \draw[thick,smooth] (2.75,-0.05) -- (2.75,0.05) node[below] {$z_{5}$};  
  \draw[thick,smooth] (3.25,-0.05) -- (3.25,0.05) node[below] {$z_{6}$};  
  \draw[thick,smooth] (3.5,-0.05) -- (3.5,0.05) node[below] {$1$};   
    \draw[thick,smooth] (-0.05,2) -- (0.05,2) node[left] {$1$};  
    \draw[thick,smooth] (-0.05,-2) -- (0.05,-2) node[left] {$-1$};  
  \draw[thick,domain=0:0.25,smooth,variable=\a] plot ({\a},2);
  \draw[thick,domain=0.31:1,smooth,variable=\b]  plot ({\b},-2);
  \draw[thick,domain=1.06:1.4,smooth,variable=\c]  plot ({\c},2);
  \draw[thick,domain=1.46:2,smooth,variable=\d]  plot ({\d},-2);
  \draw[thick,domain=2.06:2.75,smooth,variable=\e]  plot ({\e},2);
  \draw[thick,domain=2.81:3.25,smooth,variable=\f]  plot ({\f},-2);
  \draw[thick,domain=3.31:3.5,smooth,variable=\g]  plot ({\g},2);
  \filldraw (0.25,2) circle (1.4pt);
  \draw (0.25,-2) circle (1.4pt);
  \filldraw (1,-2) circle (1.4pt);
  \draw (1,2) circle (1.4pt);
  \filldraw (1.4,2) circle (1.4pt);
  \draw (1.4,-2) circle (1.4pt);
  \filldraw (2,-2) circle (1.4pt);
  \draw (2,2) circle (1.4pt);
  \filldraw (2.75,2) circle (1.4pt);
  \draw (2.75,-2) circle (1.4pt);
  \filldraw (3.25,-2) circle (1.4pt);
  \draw (3.25,2) circle (1.4pt);
  \draw[dashed,smooth] (0.25,0) -- (0.25,1.94);
  \draw[dashed,smooth] (0.25,-0.35) -- (0.25,-1.94);
  \draw[dashed,smooth] (1,0) -- (1,1.94);
  \draw[dashed,smooth] (1,-0.35) -- (1,-1.94);
  \draw[dashed,smooth] (1.4,0) -- (1.4,1.94);
  \draw[dashed,smooth] (1.4,-0.35) -- (1.4,-1.94);
  \draw[dashed,smooth] (2,0) -- (2,1.94);
  \draw[dashed,smooth] (2,-0.35) -- (2,-1.94);
  \draw[dashed,smooth] (2.75,0) -- (2.75,1.94);
  \draw[dashed,smooth] (2.75,-0.35) -- (2.75,-1.94);
  \draw[dashed,smooth] (3.25,0) -- (3.25,1.94);
  \draw[dashed,smooth] (3.25,-0.35) -- (3.25,-1.94);
\end{tikzpicture}
\qquad
\begin{tikzpicture}
  \draw[->] (-0.1,0) -- (4,0);
  \draw[->] (0,-2.75) -- (0,2.75);
  \draw[thick,smooth] (0.25,-0.05) -- (0.25,0.05) node[below] {$z_{1}$};  
  \draw[thick,smooth] (1,-0.05) -- (1,0.05) node[below] {$z_{2}$};  
  \draw[thick,smooth] (1.4,-0.05) -- (1.4,0.05) node[below] {$z_{3}$};  
  \draw[thick,smooth] (2,-0.05) -- (2,0.05) node[below] {$z_{4}$};  
  \draw[thick,smooth] (2.75,-0.05) -- (2.75,0.05) node[below] {$z_{5}$};  
  \draw[thick,smooth] (3.25,-0.05) -- (3.25,0.05) node[below] {$z_{6}$};  
  \draw[thick,smooth] (3.5,-0.05) -- (3.5,0.05) node[below] {$1$};   
    \draw[thick,smooth] (-0.05,2) -- (0.05,2) node[left] {$1$};  
    \draw[thick,smooth] (-0.05,-2) -- (0.05,-2) node[left] {$-1$}; 
  \draw[thick,domain=0:0.25,smooth,variable=\a] plot ({\a},1.898);
  \draw[thick,domain=0.31:1,smooth,variable=\b]  plot ({\b},1.777);
  \draw[thick,domain=1.06:1.4,smooth,variable=\c]  plot ({\c},-1.898);
  \draw[thick,domain=1.46:2,smooth,variable=\d]  plot ({\d},0.222);
  \draw[thick,domain=2.06:2.75,smooth,variable=\e]  plot ({\e},0);
  \draw[thick,domain=2.81:3.25,smooth,variable=\f]  plot ({\f},-2);
  \draw[thick,domain=3.31:3.5,smooth,variable=\g]  plot ({\g},1.898);
  \filldraw (0.25,1.898) circle (1.4pt);
  \draw (0.25,1.777) circle (1.4pt);
  \filldraw (1,1.777) circle (1.4pt);
  \draw (1,-1.898) circle (1.4pt);
  \filldraw (1.4,-1.898) circle (1.4pt);
  \draw (1.4,0.222) circle (1.4pt);
  \filldraw (2,0.222) circle (1.4pt);
  \draw (2,0) circle (1.4pt);
  \filldraw (2.75,0) circle (1.4pt);
  \draw (2.75,-2) circle (1.4pt);
  \filldraw (3.25,-2) circle (1.4pt);
  \draw (3.25,1.898) circle (1.4pt);
  \draw[dashed,smooth] (0.25,0) -- (0.25,1.717);
  \draw[dashed,smooth] (1,-0.35) -- (1,-1.838);
  \draw[dashed,smooth] (1,0) -- (1,1.717);
  \draw[dashed,smooth] (1.4,-0.35) -- (1.4,-1.838);
  \draw[dashed,smooth] (1.4,0) -- (1.4,0.162);
  \draw[dashed,smooth] (2,0) -- (2,0.162);
  \draw[dashed,smooth] (2.75,-0.35) -- (2.75,-1.94);
  \draw[dashed,smooth] (3.25,-0.35) -- (3.25,-1.94);
  \draw[dashed,smooth] (3.25,0) -- (3.25,1.717);
\end{tikzpicture}
\qquad
\begin{tikzpicture}
  \draw[->] (-0.1,0) -- (4,0);
   \draw[->] (0,-2.75) -- (0,2.75);
  \draw[thick,smooth] (0.25,-0.05) -- (0.25,0.05) node[below] {$z_{1}$};  
  \draw[thick,smooth] (1,-0.05) -- (1,0.05) node[below] {$z_{2}$};  
  \draw[thick,smooth] (1.4,-0.05) -- (1.4,0.05) node[above] {$z_{3}$};  
  \draw[thick,smooth] (2,-0.05) -- (2,0.05) node[above] {$z_{4}$};  
  \draw[thick,smooth] (2.75,-0.05) -- (2.75,0.05) node[below] {$z_{5}$};  
  \draw[thick,smooth] (3.25,-0.05) -- (3.25,0.05) node[below] {$z_{6}$};  
  \draw[thick,smooth] (3.5,-0.05) -- (3.5,0.05) node[below] {$1$};   
    \draw[thick,smooth] (-0.05,2) -- (0.05,2) node[left] {$1$};  
    \draw[thick,smooth] (-0.05,-2) -- (0.05,-2) node[left] {$-1$}; 
  \draw[thick,domain=0:0.25,smooth,variable=\a] plot ({\a},1.898);
  \draw[thick,domain=0.31:1,smooth,variable=\b]  plot ({\b},-1.777);
  \draw[thick,domain=1.06:1.4,smooth,variable=\c]  plot ({\c},-1.898);
  \draw[thick,domain=1.46:2,smooth,variable=\d]  plot ({\d},-0.222);
  \draw[thick,domain=2.06:2.75,smooth,variable=\e]  plot ({\e},0);
  \draw[thick,domain=2.81:3.25,smooth,variable=\f]  plot ({\f},2);
  \draw[thick,domain=3.31:3.5,smooth,variable=\g]  plot ({\g},1.898);
  \filldraw (0.25,1.898) circle (1.4pt);
  \draw (0.25,-1.777) circle (1.4pt);
  \filldraw (1,-1.777) circle (1.4pt);
  \draw (1,-1.898) circle (1.4pt);
  \filldraw (1.4,-1.898) circle (1.4pt);
  \draw (1.4,-0.222) circle (1.4pt);
  \filldraw (2,-0.222) circle (1.4pt);
  \draw (2,0) circle (1.4pt);
  \filldraw (2.75,0) circle (1.4pt);
  \draw (2.75,2) circle (1.4pt);
  \filldraw (3.25,2) circle (1.4pt);
  \draw (3.25,1.898) circle (1.4pt);
  \draw[dashed,smooth] (0.25,0) -- (0.25,1.838);
  \draw[dashed,smooth] (0.25,-0.35) -- (0.25,-1.717);
  \draw[dashed,smooth] (1,-0.35) -- (1,-1.838);
  \draw[dashed,smooth] (1.4,0) -- (1.4,-0.162);
  \draw[dashed,smooth] (1.4,-0.282) -- (1.4,-1.838);
  \draw[dashed,smooth] (2,-0.06) -- (2,-0.162);
  \draw[dashed,smooth] (2.75,0) -- (2.75,1.94);
  \draw[dashed,smooth] (3.25,0) -- (3.25,1.838);
\end{tikzpicture}}
\vspace{0.5em}
\caption{Eigenfunctions $f_{0}$, $f_{1}$, $f_{2}$, $f_{3}$, $f_{4}$ and $f_{5}$ of $\Delta^{\mu}$ for $\mu = \sum_{i=1}^{6} \alpha_{i} \delta_{z_i}$ with $\alpha_{1} = \alpha_{3} = \alpha_{5}= 1/4$ and $\alpha_{2}=\alpha_{4}=\alpha_{6}= 1/12$. Compare with Figure 4.}
\label{fig:Counterexapmle}
\end{figure}

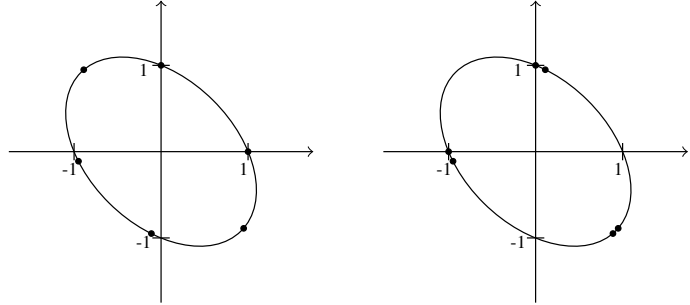
\begin{figure}[ht!]
	\centering
\resizebox{0.715\linewidth}{!}{
\begin{tikzpicture}
  	\draw[->] (-1.75,0) -- (1.75,0);
 	\draw[->] (0,-1.75) -- (0,1.75);
	  \filldraw (-0.9493337495, -0.1111111111) circle (0.9pt);
	  \filldraw (-0.8888888889, 0.9493337495) circle (0.9pt);
	  \filldraw (0.9493337495, -0.8888888889) circle (0.9pt);
  	  \filldraw (-0.1111111111, -0.9493337495) circle (0.9pt);
    	  \filldraw (0,1) circle (0.9pt);
    	  \filldraw (1,0) circle (0.9pt);
	  \draw[rotate=45] (0,0) ellipse (24pt and 37pt);
	\draw[smooth] (1,-0.1) -- (1,0.1);
	\draw[smooth] (-1,-0.1) -- (-1,0.1);
	\draw[smooth] (-0.1,1) -- (0.1,1);
	\draw[smooth] (-0.1,-1) -- (0.1,-1);
	\node at (0.95,-0.2) {\tiny 1};
	\node at (-1.05,-0.2) {\tiny -1};
	\node at (-0.2,0.95) {\tiny 1};
	\node at (-0.2,-1.05) {\tiny -1};
\end{tikzpicture}
\qquad
\begin{tikzpicture}
  	\draw[->] (-1.75,0) -- (1.75,0);
 	\draw[->] (0,-1.75) -- (0,1.75);
	  \filldraw (-0.9493337495, -0.1111111111) circle (0.9pt);
	  \filldraw (0.8888888889, -0.9493337495) circle (0.9pt);
	  \filldraw (0.9493337495, -0.8888888889) circle (0.9pt);
  	  \filldraw (0.1111111111, 0.9493337495) circle (0.9pt);
    	  \filldraw (0,1) circle (0.9pt);
    	  \filldraw (-1,0) circle (0.9pt);
	  \draw[rotate=45] (0,0) ellipse (24pt and 37pt);
	\draw[smooth] (1,-0.1) -- (1,0.1);
	\draw[smooth] (-1,-0.1) -- (-1,0.1);
	\draw[smooth] (-0.1,1) -- (0.1,1);
	\draw[smooth] (-0.1,-1) -- (0.1,-1);
	\node at (0.95,-0.2) {\tiny 1};
	\node at (-1.05,-0.2) {\tiny -1};
	\node at (-0.2,0.95) {\tiny 1};
	\node at (-0.2,-1.05) {\tiny -1};
\end{tikzpicture}}
	\caption{Point plot of $S_{1,5}$ (left) and point plot of $S_{2,4}$ (right) together with the curve given by $\sqrt{73}(x^2+y^2-1) = -7xy$.}
	\label{fig:Ellipses}
\end{figure}

\section{Uniform discrete probability distributions}\label{sec:examples}

Here, we consider the case when $\mu$ is a uniform discrete probability distribution with $N \geq 3$, namely $\alpha_i = N^{-1}$ for all $i \in \{ 1,\dots, N \}$, in which case,
\begin{align*}
B = 
\begin{pmatrix}
-2N^{2} & N^{2} & 0 & \cdots & 0 & 0 & N^{2} \\[0.25em]
N^{2} & -2N^{2} &  N^{2} & \cdots & 0 & 0 & 0\\[0.25em]
0 & N^{2} & -2N^{2} &  \cdots & 0 & 0 & 0\\[0.25em]
\vdots & \vdots & \vdots & \ddots & \vdots & \vdots & \vdots \\[0.25em]
0 & 0 &0 &  \cdots & -2N^{2} & N^{2} & 0\\
0 & 0 & 0 &  \cdots & N^{2} & -2N^{2} &  N^{2}\\[0.25em]
N^{2} & 0 &  0 & \cdots & 0 & N^{2} & -2N^{2}
\end{pmatrix}.
\end{align*}
The following theorem reveals the spectrum of this matrix, and in the case that $N$ is even, the result gives a second example for which the lower bound in Theorem \ref{prop:lower_bound} is sharp.

\begin{theorem}\label{prop:ev_all_equal}
The eigenvalues of the matrix $B$ given directly above are of the form
\begin{align*}
\lambda_l = -2N^{2} + 2N^{2} \cos \left( 2 \pi l / N \right)
\end{align*}
with corresponding eigenvectors
\begin{align*}
v^{(l)}=\bigg(1, \exp\left(2\pi \ii l /N\right),\exp\left(2\pi \ii 2l /N\right),\dots,\exp\left(2\pi \ii (N-1) l /N\right)\bigg)^{\top}\hspace{-0.5em},
\end{align*}
for $l \in \{0,\dots,N-1\}$.
\end{theorem}

\begin{proof}
Set $m_{1} = -2N^{2}, m_{2} = m_{N} = N^{2}$ and $m_i = 0$, for $i \in \{ 3, \dots, N-1 \}$. The identity $B v = \lambda v$ can be formulated as a system of $N$ difference equations of the form
\begin{align}\label{eq:difference_eq}
\sum_{k=1}^{N-j} m_k v_{k+j} + \sum_{k=N-j+1}^{N} m_k v_{k-N+j} = \lambda v_{j+1},
\end{align}
where $j \in \{0,\dots, N-1\}$ and $v=(v_1, \dots, v_{N})$. To obtain the eigenvalue $\lambda_{l}$, we follow the ansatz $v_{k}^{(l)} = \varphi_l^{k-1}$, where $\varphi_{l}^{k} \coloneqq \exp(2\pi \ii k l / N)$, for $k, l \in \{0, \dots, N-1 \}$. Substituting this into \eqref{eq:difference_eq}, and using the facts that $\varphi_l^{-N} = \varphi_l^{0} = 1$ and $\varphi_l^{j} \neq 0$, for all $j,l\in\{ 0,\dots, N-1\}$, we obtain $\lambda_{l} = \sum_{k=1}^{N} m_{k} \varphi_{l}^{k-1}$.  Hence, for $l \in \{0,\dots,N-1\}$, we have $B v^{(l)} = \lambda_l v^{(l)}$ and
\begin{align*}
\lambda_{l} 
&= -2N^{2} + N^{2} \exp(2\pi \ii l /N) + N^{2} \exp(2\pi \ii l (N-1)/N)\\
&= -2N^{2} + 2N^{2} \cos \left( 2 \pi l /N \right).\qedhere
\end{align*}
\end{proof}
\begin{corollary}\label{prop:ev_and_ef_of_Delta_all_equal}
The eigenvalues of the operator $\Delta^{\mu}$ are $\lambda_l = -2N^{2} + 2N^{2} \cos \left( 2 \pi l / N \right)$, for $l \in \{ 0,\dots, N-1\}$, with corresponding eigenfunctions $f_{l} \in \mathscr{D}_{\mu}^{2}$, where
\begin{enumerate}
\item $f_{0}$ is the constant function with value $1$,
\end{enumerate}
and, 	for $j \in \{1,\dots,N-1\}$,
\begin{enumerate}
\setcounter{enumi}{1}
	\item $\displaystyle f_{l}\vert_{[0,z_{1}] \cup (z_{N},1]} = 0$ and $\displaystyle f_{l}\vert_{(z_{j}, z_{j+1}]} =  \operatorname{Im} \left( \exp\left(2\pi \ii j l/  N\right)\right)$, for $0  < l < N/2$, and
	\item $\displaystyle f_{l}\vert_{[0,z_{1}] \cup (z_{N},1]} = 1$ and $\displaystyle f_{l}\vert_{(z_{j}, z_{j+1}]} =  \operatorname{Re} \left( \exp\left(2\pi \ii j l /N\right)\right)$, for $N/2 \leq l \leq N-1$. 
\end{enumerate}
(See Figures \ref{fig:uniform_N=3} and \ref{fig:uniform_N=6}.)
\end{corollary}

In the situation of Corollary~\ref{prop:ev_and_ef_of_Delta_all_equal}, we have for $N$ tending to infinity, that the pure point measures converge weakly to the Lebesgue measure, and indeed, the eigenfunctions of the discrete Laplacians converge uniformly to (appropriately re-scaled) cosine and sine functions.  Moreover,  the eigenvalue approach $-(2\pi k)^2$; the corresponding eigenvalues of the classical weak Laplacian.

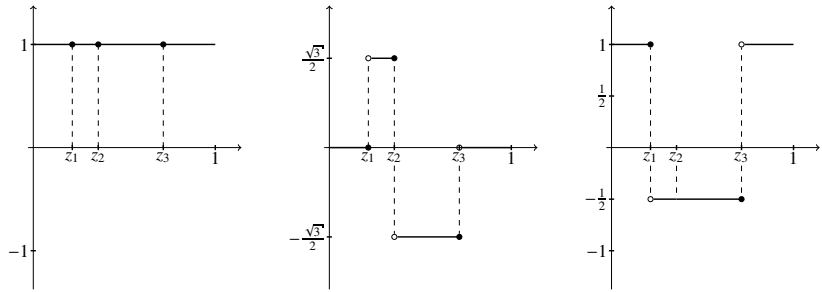
\begin{figure}[ht!]
\centering
\resizebox{0.85\linewidth}{!}{%
\begin{tikzpicture}
  \draw[->] (-0.1,0) -- (4,0); 
  \draw[->] (0,-2.75) -- (0,2.75); 
  \draw[thick,smooth] (0.75,-0.05) -- (0.75,0.05) node[below] {$z_{1}$};  
  \draw[thick,smooth] (1.25,-0.05) -- (1.25,0.05) node[below] {$z_{2}$};  
  \draw[thick,smooth] (2.5,-0.05) -- (2.5,0.05) node[below] {$z_{3}$};  
  \draw[thick,smooth] (3.5,-0.05) -- (3.5,0.05) node[below] {$1$};   
    \draw[thick,smooth] (-0.05,2) -- (0.05,2) node[left] {$1$};  
    \draw[thick,smooth] (-0.05,-2) -- (0.05,-2) node[left] {$-1$}; 
  \draw[thick,domain=0:0.75,smooth,variable=\w] plot ({\w},2);
  \draw[thick,domain=0.75:1.25,smooth,variable=\x]  plot ({\x},2);
  \draw[thick,domain=1.25:2.5,smooth,variable=\y]  plot ({\y},2);
  \draw[thick,domain=2.5:3.5,smooth,variable=\z]  plot ({\z},2);
%
  \draw[dashed,smooth] (0.75,-0) -- (0.75,2);
  \draw[dashed,smooth] (1.25,0) -- (1.25,2);
  \draw[dashed,smooth] (2.5,0) -- (2.5,2);
%
  \filldraw (0.75,2) circle (1.4pt);
  \filldraw (1.25,2) circle (1.4pt);
  \filldraw (2.5,2) circle (1.4pt);
\end{tikzpicture}
\qquad
\begin{tikzpicture}
  \draw[->] (-0.1,0) -- (4,0); 
  \draw[->] (0,-2.75) -- (0,2.75);
  \draw[thick,smooth] (0.75,-0.05) -- (0.75,0.05) node[below] {$z_{1}$};  
  \draw[thick,smooth] (1.25,-0.05) -- (1.25,0.05) node[below] {$z_{2}$};  
  \draw[thick,smooth] (2.5,-0.05) -- (2.5,0.05) node[below] {$z_{3}$};  
  \draw[thick,smooth] (3.5,-0.05) -- (3.5,0.05) node[below] {$1$};   
    \draw[thick,smooth] (-0.05,1.732) -- (0.05,1.732) node[left] {$\frac{\sqrt{3}}{2}$};  
    \draw[thick,smooth] (-0.05,-1.732) -- (0.05,-1.732) node[left] {$-\frac{\sqrt{3}}{2}$}; 
  \draw[thick,domain=0:0.75,smooth,variable=\w] plot ({\w},0);
  \draw[thick,domain=0.81:1.25,smooth,variable=\x]  plot ({\x},1.732);
  \draw[thick,domain=1.31:2.5,smooth,variable=\y]  plot ({\y},-1.732);
  \draw[thick,domain=2.56:3.5,smooth,variable=\z]  plot ({\z},0);
  \filldraw (0.75,0) circle (1.4pt);
  \draw (0.75,1.732) circle (1.4pt);
  \filldraw (1.25,1.732) circle (1.4pt);
  \draw (1.25,-1.732) circle (1.4pt);
  \filldraw (2.5,-1.732) circle (1.4pt);
  \draw (2.5,0) circle (1.4pt);
  \draw[dashed,smooth] (0.75,0) -- (0.75,1.672);
  \draw[dashed,smooth] (1.25,-0.35) -- (1.25,-1.672);
  \draw[dashed,smooth] (1.25,0) -- (1.25,1.672);
  \draw[dashed,smooth] (2.5,-0.35) -- (2.5,-1.672);
\end{tikzpicture}
\qquad
\begin{tikzpicture}
  \draw[->] (-0.1,0) -- (4,0);
  \draw[->] (0,-2.75) -- (0,2.75);
  \draw[thick,smooth] (0.75,-0.05) -- (0.75,0.05) node[below] {$z_{1}$};  
  \draw[thick,smooth] (1.25,-0.05) -- (1.25,0.05) node[below] {$z_{2}$};  
  \draw[thick,smooth] (2.5,-0.05) -- (2.5,0.05) node[below] {$z_{3}$};  
  \draw[thick,smooth] (3.5,-0.05) -- (3.5,0.05) node[below] {$1$};   
    \draw[thick,smooth] (-0.05,2) -- (0.05,2) node[left] {$1$};  
    \draw[thick,smooth] (-0.05,-2) -- (0.05,-2) node[left] {$-1$};  
    \draw[thick,smooth] (-0.05,1) -- (0.05,1) node[left] {$\frac{1}{2}$}; 
    \draw[thick,smooth] (-0.05,-1) -- (0.05,-1) node[left] {$-\frac{1}{2}$}; 
  \draw[thick,domain=0:0.75,smooth,variable=\w] plot ({\w},2);
  \draw[thick,domain=0.81:1.25,smooth,variable=\x]  plot ({\x},-1);
  \draw[thick,domain=1.25:2.5,smooth,variable=\y]  plot ({\y},-1);
  \draw[thick,domain=2.56:3.5,smooth,variable=\z]  plot ({\z},2);
  \filldraw (0.75,2) circle (1.4pt);
  \draw (0.75,-1) circle (1.4pt);
  \filldraw (2.5,-1) circle (1.4pt);
  \draw (2.5,2) circle (1.4pt);
  \draw[dashed,smooth] (0.75,-0.35) -- (0.75,-0.94);
  \draw[dashed,smooth] (0.75,0) -- (0.75,1.94);
  \draw[dashed,smooth] (2.5,-0.35) -- (2.5,-0.94);
  \draw[dashed,smooth] (2.5,0) -- (2.5,1.94);
  \draw[dashed,smooth] (1.25,-0.35) -- (1.25,-0.94);
\end{tikzpicture}}
\vspace{0.5em}
\caption{Eigenfunctions $f_0$, $f_1$ and $f_2$ of $\Delta^{\mu}$ with corresponding eigenvalues $\lambda_{0} = 0$, $\lambda_{1} = -27$ and $\lambda_{2} = -27$, for $\mu$ a uniform discrete probability distribution with $N = 3$.
\label{fig:uniform_N=3}
}
\end{figure}
\begin{figure}[ht!]
\centering
\resizebox{0.85\linewidth}{!}{%
\begin{tikzpicture}
  \draw[->] (-0.1,0) -- (4,0);
  \draw[->] (0,-2.75) -- (0,2.75);
  \draw[thick,smooth] (0.25,-0.05) -- (0.25,0.05) node[below] {$z_{1}$};  
  \draw[thick,smooth] (1,-0.05) -- (1,0.05) node[below] {$z_{2}$};  
  \draw[thick,smooth] (1.4,-0.05) -- (1.4,0.05) node[below] {$z_{3}$};  
  \draw[thick,smooth] (2,-0.05) -- (2,0.05) node[below] {$z_{4}$};  
  \draw[thick,smooth] (2.75,-0.05) -- (2.75,0.05) node[below] {$z_{5}$};  
  \draw[thick,smooth] (3.25,-0.05) -- (3.25,0.05) node[below] {$z_{6}$};  
  \draw[thick,smooth] (3.5,-0.05) -- (3.5,0.05) node[below] {$1$};   
    \draw[thick,smooth] (-0.05,2) -- (0.05,2) node[left] {$1$};  
    \draw[thick,smooth] (-0.05,-2) -- (0.05,-2) node[left] {$-1$}; 
  \draw[thick,domain=0:0.75,smooth,variable=\w] plot ({\w},2);
  \draw[thick,domain=0.75:1.25,smooth,variable=\x]  plot ({\x},2);
  \draw[thick,domain=1.25:2.5,smooth,variable=\y]  plot ({\y},2);
  \draw[thick,domain=2.5:3.5,smooth,variable=\z]  plot ({\z},2);
%
  \draw[dashed,smooth] (0.25,-0) -- (0.25,2);  
  \draw[dashed,smooth] (1,0) -- (1,2);
  \draw[dashed,smooth] (1.4,0) -- (1.4,2);  
  \draw[dashed,smooth] (2,0) -- (2,2);
  \draw[dashed,smooth] (2.75,0) -- (2.75,2);
  \draw[dashed,smooth] (3.25,0) -- (3.25,2);
%
  \filldraw (0.25,2) circle (1.4pt);
  \filldraw (1,2) circle (1.4pt);
  \filldraw (1.4,2) circle (1.4pt);
  \filldraw (2,2) circle (1.4pt);
  \filldraw (2.75,2) circle (1.4pt);
  \filldraw (3.25,2) circle (1.4pt);
\end{tikzpicture}
\qquad
\begin{tikzpicture}
  \draw[->] (-0.1,0) -- (4,0); 
  \draw[->] (0,-2.75) -- (0,2.75);
  \draw[thick,smooth] (0.25,-0.05) -- (0.25,0.05) node[below] {$z_{1}$};  
  \draw[thick,smooth] (1,-0.05) -- (1,0.05) node[below] {$z_{2}$};  
  \draw[thick,smooth] (1.4,-0.05) -- (1.4,0.05) node[below] {$z_{3}$};  
  \draw[thick,smooth] (2,-0.05) -- (2,0.05) node[below] {$z_{4}$};  
  \draw[thick,smooth] (2.75,-0.05) -- (2.75,0.05) node[below] {$z_{5}$};  
  \draw[thick,smooth] (3.25,-0.05) -- (3.25,0.05) node[below] {$z_{6}$};  
  \draw[thick,smooth] (3.5,-0.05) -- (3.5,0.05) node[below] {$1$};  
    \draw[thick,smooth] (-0.05,-1.732) -- (0.05,-1.732) node[left] {$-\frac{\sqrt{3}}{2}$};  
    \draw[thick,smooth] (-0.05,1.732) -- (0.05,1.732) node[left] {$\frac{\sqrt{3}}{2}$}; 
  \draw[thick,domain=0:0.25,smooth,variable=\a] plot ({\a},0);
  \draw[thick,domain=0.32:1,smooth,variable=\b]  plot ({\b},1.732);
  \draw[thick,domain=1:1.4,smooth,variable=\c]  plot ({\c},1.732);
  \draw[thick,domain=1.46:2,smooth,variable=\d]  plot ({\d},0);
  \draw[thick,domain=2.06:2.75,smooth,variable=\e]  plot ({\e},-1.732);
  \draw[thick,domain=2.75:3.25,smooth,variable=\f]  plot ({\f},-1.732);
  \draw[thick,domain=3.29:3.5,smooth,variable=\g]  plot ({\g},0);
  \filldraw (0.25,0) circle (1.4pt);
  \draw (0.25,1.732) circle (1.4pt);
  \filldraw (1.4,1.732) circle (1.4pt);
  \draw (1.4,0) circle (1.4pt);
  \filldraw (2,0) circle (1.4pt);
  \draw (2,-1.732) circle (1.4pt);
  \filldraw (3.25,-1.732) circle (1.4pt);
  \draw (3.25,0) circle (1.4pt);
  \draw[dashed,smooth] (0.25,0) -- (0.25,1.672);
  \draw[dashed,smooth] (1.4,1.672) -- (1.4,0);
  \draw[dashed,smooth] (2,-0.35) -- (2,-1.672);
  \draw[dashed,smooth] (3.25,-0.35) -- (3.25,-1.672);
  \draw[dashed,smooth] (1,0) -- (1,1.672);
  \draw[dashed,smooth] (2.75,-0.35) -- (2.75,-1.672);
\end{tikzpicture}
\qquad
\begin{tikzpicture}
  \draw[->] (-0.1,0) -- (4,0);
  \draw[->] (0,-2.75) -- (0,2.75);
  \draw[thick,smooth] (0.25,-0.05) -- (0.25,0.05) node[below] {$z_{1}$};  
  \draw[thick,smooth] (1,-0.05) -- (1,0.05) node[below] {$z_{2}$};  
  \draw[thick,smooth] (1.4,-0.05) -- (1.4,0.05) node[below] {$z_{3}$};  
  \draw[thick,smooth] (2,-0.05) -- (2,0.05) node[below] {$z_{4}$};  
  \draw[thick,smooth] (2.75,-0.05) -- (2.75,0.05) node[below] {$z_{5}$};  
  \draw[thick,smooth] (3.25,-0.05) -- (3.25,0.05) node[below] {$z_{6}$};  
  \draw[thick,smooth] (3.5,-0.05) -- (3.5,0.05) node[below] {$1$};  
    \draw[thick,smooth] (-0.05,1.732) -- (0.05,1.732) node[left] {$\frac{\sqrt{3}}{2}$};  
    \draw[thick,smooth] (-0.05,-1.732) -- (0.05,-1.732) node[left] {$-\frac{\sqrt{3}}{2}$}; 
  \draw[thick,domain=0:0.25,smooth,variable=\a] plot ({\a},0);
  \draw[thick,domain=0.31:1,smooth,variable=\b]  plot ({\b},1.732);
  \draw[thick,domain=1.06:1.4,smooth,variable=\c]  plot ({\c},-1.732);
  \draw[thick,domain=1.46:2,smooth,variable=\d]  plot ({\d},0);
  \draw[thick,domain=2.06:2.75,smooth,variable=\e]  plot ({\e},1.732);
  \draw[thick,domain=2.81:3.25,smooth,variable=\f]  plot ({\f},-1.732);
  \draw[thick,domain=3.31:3.5,smooth,variable=\g]  plot ({\g},0);
  \filldraw (0.25,0) circle (1.4pt);
  \draw (0.25,1.732) circle (1.4pt);
  \filldraw (1,1.732) circle (1.4pt);
  \draw (1,-1.732) circle (1.4pt);
  \filldraw (1.4,-1.732) circle (1.4pt);
  \draw (1.4,0) circle (1.4pt);
  \filldraw (2,0) circle (1.4pt);
  \draw (2,1.732) circle (1.4pt);
  \filldraw (2.75,1.732) circle (1.4pt);
  \draw (2.75,-1.732) circle (1.4pt);
  \filldraw (3.25,-1.732) circle (1.4pt);
  \draw (3.25,0) circle (1.4pt);
  \draw[dashed,smooth] (0.25,0) -- (0.25,1.672);
  \draw[dashed,smooth] (1,1.672) -- (1,0);
  \draw[dashed,smooth] (1,-0.35) -- (1,-1.672);
  \draw[dashed,smooth] (1.4,-0.35) -- (1.4,-1.672);
  \draw[dashed,smooth] (2,0) -- (2,1.672);
  \draw[dashed,smooth] (2.75,-0.35) -- (2.75,-1.672);
  \draw[dashed,smooth] (2.75,0) -- (2.75,1.672);
  \draw[dashed,smooth] (3.25,-0.35) -- (3.25,-1.672);
\end{tikzpicture}}
\end{figure}

\begin{figure}[ht!]
\centering
\resizebox{0.85\linewidth}{!}{%
\begin{tikzpicture}
  \draw[->] (-0.1,0) -- (4,0);
  \draw[->] (0,-2.75) -- (0,2.75);
  \draw[thick,smooth] (0.25,-0.05) -- (0.25,0.05) node[below] {$z_{1}$};  
  \draw[thick,smooth] (1,-0.05) -- (1,0.05) node[below] {$z_{2}$};  
  \draw[thick,smooth] (1.4,-0.05) -- (1.4,0.05) node[below] {$z_{3}$};  
  \draw[thick,smooth] (2,-0.05) -- (2,0.05) node[below] {$z_{4}$};  
  \draw[thick,smooth] (2.75,-0.05) -- (2.75,0.05) node[below] {$z_{5}$};  
  \draw[thick,smooth] (3.25,-0.05) -- (3.25,0.05) node[below] {$z_{6}$};  
  \draw[thick,smooth] (3.5,-0.05) -- (3.5,0.05) node[below] {$1$};  
    \draw[thick,smooth] (-0.05,2) -- (0.05,2) node[left] {$1$};  
    \draw[thick,smooth] (-0.05,-2) -- (0.05,-2) node[left] {$-1$};  
  \draw[thick,domain=0:0.25,smooth,variable=\a] plot ({\a},2);
  \draw[thick,domain=0.31:1,smooth,variable=\b]  plot ({\b},-2);
  \draw[thick,domain=1.06:1.4,smooth,variable=\c]  plot ({\c},2);
  \draw[thick,domain=1.46:2,smooth,variable=\d]  plot ({\d},-2);
  \draw[thick,domain=2.06:2.75,smooth,variable=\e]  plot ({\e},2);
  \draw[thick,domain=2.81:3.25,smooth,variable=\f]  plot ({\f},-2);
  \draw[thick,domain=3.31:3.5,smooth,variable=\g]  plot ({\g},2);
  \filldraw (0.25,2) circle (1.4pt);
  \draw (0.25,-2) circle (1.4pt);
  \filldraw (1,-2) circle (1.4pt);
  \draw (1,2) circle (1.4pt);
  \filldraw (1.4,2) circle (1.4pt);
  \draw (1.4,-2) circle (1.4pt);
  \filldraw (2,-2) circle (1.4pt);
  \draw (2,2) circle (1.4pt);
  \filldraw (2.75,2) circle (1.4pt);
  \draw (2.75,-2) circle (1.4pt);
  \filldraw (3.25,-2) circle (1.4pt);
  \draw (3.25,2) circle (1.4pt);
  \draw[dashed,smooth] (0.25,-0.35) -- (0.25,-1.94);
  \draw[dashed,smooth] (0.25,0) -- (0.25,1.94);
  \draw[dashed,smooth] (1,-0.35) -- (1,-1.94);
  \draw[dashed,smooth] (1,0) -- (1,1.94);
  \draw[dashed,smooth] (1.4,-0.35) -- (1.4,-1.94);
  \draw[dashed,smooth] (1.4,0) -- (1.4,1.94);
  \draw[dashed,smooth] (2,-0.35) -- (2,-1.94);
  \draw[dashed,smooth] (2,0) -- (2,1.94);
  \draw[dashed,smooth] (2.75,-0.35) -- (2.75,-1.94);
  \draw[dashed,smooth] (2.75,0) -- (2.75,1.94);
  \draw[dashed,smooth] (3.25,-0.35) -- (3.25,-1.94);
  \draw[dashed,smooth] (3.25,0) -- (3.25,1.94);
\end{tikzpicture}
\qquad
\begin{tikzpicture}
  \draw[->] (-0.1,0) -- (4,0);
  \draw[->] (0,-2.75) -- (0,2.75);
  \draw[thick,smooth] (0.25,-0.05) -- (0.25,0.05) node[below] {$z_{1}$};  
  \draw[thick,smooth] (1,-0.05) -- (1,0.05) node[below] {$z_{2}$};  
  \draw[thick,smooth] (1.4,-0.05) -- (1.4,0.05) node[below] {$z_{3}$};  
  \draw[thick,smooth] (2,-0.05) -- (2,0.05) node[below] {$z_{4}$};  
  \draw[thick,smooth] (2.75,-0.05) -- (2.75,0.05) node[below] {$z_{5}$};  
  \draw[thick,smooth] (3.25,-0.05) -- (3.25,0.05) node[below] {$z_{6}$};  
  \draw[thick,smooth] (3.5,-0.05) -- (3.5,0.05) node[below] {$1$};
    \draw[thick,smooth] (-0.05,2) -- (0.05,2) node[left] {$1$};  
    \draw[thick,smooth] (-0.05,-2) -- (0.05,-2) node[left] {$-1$}; 
    \draw[thick,smooth] (-0.05,1) -- (0.05,1) node[left] {$\frac{1}{2}$};  
    \draw[thick,smooth] (-0.05,-1) -- (0.05,-1) node[left] {$-\frac{1}{2}$}; 
   \draw[thick,smooth] (-0.05,-1.732)  (0.05,-1.732) node[left] {\phantom{$-\frac{\sqrt{3}}{2}$}}; 
  \draw[thick,domain=0:0.25,smooth,variable=\a] plot ({\a},2);
  \draw[thick,domain=0.31:1,smooth,variable=\b]  plot ({\b},-1);
  \draw[thick,domain=1:1.4,smooth,variable=\c]  plot ({\c},-1);
  \draw[thick,domain=1.46:2,smooth,variable=\d]  plot ({\d},2);
  \draw[thick,domain=2.06:2.75,smooth,variable=\e]  plot ({\e},-1);
  \draw[thick,domain=2.75:3.25,smooth,variable=\f]  plot ({\f},-1);
  \draw[thick,domain=3.31:3.5,smooth,variable=\g]  plot ({\g},2);
  \filldraw (0.25,2) circle (1.4pt);
  \draw (0.25,-1) circle (1.4pt);
  \filldraw (1.4,-1) circle (1.4pt);
  \draw (1.4,2) circle (1.4pt);
  \filldraw (2,2) circle (1.4pt);
  \draw (2,-1) circle (1.4pt);
  \filldraw (3.25,-1) circle (1.4pt);
  \draw (3.25,2) circle (1.4pt);
  \draw[dashed,smooth] (0.25,0) -- (0.25,1.94);
  \draw[dashed,smooth] (0.25,-0.35) -- (0.25,-0.94);
  \draw[dashed,smooth] (1.4,-0.35) -- (1.4,-0.94);
  \draw[dashed,smooth] (1.4,0) -- (1.4,1.94);
  \draw[dashed,smooth] (2,1.94) -- (2,0);
  \draw[dashed,smooth] (2,-0.35) -- (2,-0.94);
  \draw[dashed,smooth] (3.25,-0.35) -- (3.25,-0.94);
  \draw[dashed,smooth] (3.25,0) -- (3.25,1.94);
  \draw[dashed,smooth] (1,-0.35) -- (1,-0.94);
  \draw[dashed,smooth] (2.75,-0.35) -- (2.75,-0.94);
\end{tikzpicture}
\qquad
\begin{tikzpicture}
  \draw[->] (-0.1,0) -- (4,0);
  \draw[->] (0,-2.75) -- (0,2.75);
  \draw[thick,smooth] (0.25,-0.05) -- (0.25,0.05) node[below] {$z_{1}$};  
  \draw[thick,smooth] (1,-0.05) -- (1,0.05) node[below] {$z_{2}$};  
  \draw[thick,smooth] (1.4,-0.05) -- (1.4,0.05) node[below] {$z_{3}$};  
  \draw[thick,smooth] (2,-0.05) -- (2,0.05) node[below] {$z_{4}$};  
  \draw[thick,smooth] (2.75,-0.05) -- (2.75,0.05) node[below] {$z_{5}$};  
  \draw[thick,smooth] (3.25,-0.05) -- (3.25,0.05) node[below] {$z_{6}$};  
  \draw[thick,smooth] (3.5,-0.05) -- (3.5,0.05) node[below] {$1$};  
    \draw[thick,smooth] (-0.05,2) -- (0.05,2) node[left] {$1$};  
    \draw[thick,smooth] (-0.05,-2) -- (0.05,-2) node[left] {$-1$};  
    \draw[thick,smooth] (-0.05,1) -- (0.05,1) node[left] {$\frac{1}{2}$}; 
    \draw[thick,smooth] (-0.05,-1) -- (0.05,-1) node[left] {$-\frac{1}{2}$}; 
    \draw[thick,smooth] (-0.05,-1.732)  (0.05,-1.732) node[left] {\phantom{$-\frac{\sqrt{3}}{2}$}}; 
  \draw[thick,domain=0:0.25,smooth,variable=\a] plot ({\a},2);
  \draw[thick,domain=0.31:1,smooth,variable=\b]  plot ({\b},1);
  \draw[thick,domain=1.06:1.4,smooth,variable=\c]  plot ({\c},-1);
  \draw[thick,domain=1.46:2,smooth,variable=\d]  plot ({\d},-2);
  \draw[thick,domain=2.06:2.75,smooth,variable=\e]  plot ({\e},-1);
  \draw[thick,domain=2.81:3.25,smooth,variable=\f]  plot ({\f},1);
  \draw[thick,domain=3.31:3.5,smooth,variable=\g]  plot ({\g},2);
  \filldraw (0.25,2) circle (1.4pt);
  \draw (0.25,1) circle (1.4pt);
  \filldraw (1,1) circle (1.4pt);
  \draw (1,-1) circle (1.4pt);
  \filldraw (1.4,-1) circle (1.4pt);
  \draw (1.4,-2) circle (1.4pt);
  \filldraw (2,-2) circle (1.4pt);
  \draw (2,-1) circle (1.4pt);
  \filldraw (2.75,-1) circle (1.4pt);
  \draw (2.75,1) circle (1.4pt);
  \filldraw (3.25,1) circle (1.4pt);
  \draw (3.25,2) circle (1.4pt);
  \draw[dashed,smooth] (0.25,0) -- (0.25,1.94);
  \draw[dashed,smooth] (1,-0.35) -- (1,-0.94);
  \draw[dashed,smooth] (1,0) -- (1,0.94);
  \draw[dashed,smooth] (1.4,-0.35) -- (1.4,-1.94);
  \draw[dashed,smooth] (2,-0.35) -- (2,-0.94);
  \draw[dashed,smooth] (2,-1.06) -- (2,-1.94);
  \draw[dashed,smooth] (2.75,-0.35) -- (2.75,-0.94);
  \draw[dashed,smooth] (2.75,0) -- (2.75,0.94);
  \draw[dashed,smooth] (3.25,0) -- (3.25,1.94);
\end{tikzpicture}}
\caption{Eigenfunctions $f_0$, $f_1$, $f_2$, $f_3$, $f_4$ and $f_5$ of $\Delta^{\mu}$ with corresponding eigenvalues $\lambda_{0} = 0$, $\lambda_{1} = -36$, $\lambda_{2} = -108$, $\lambda_{3} = -144$, $\lambda_{4} = -108$ and \mbox{$\lambda_{5} = -36$}, for $\mu$ a uniform discrete probability distribution with $N = 6$.}
\label{fig:uniform_N=6}
\end{figure}

\bibliographystyle{plain}
\bibliography{bib}

\end{document}